\documentclass[11pt]{elsarticle}
\journal{Journal of Mathematical Analysis and Applications}

\usepackage{amsmath,amsfonts,amssymb,amsthm,graphicx,hyperref,url}
\usepackage[labelfont=bf]{caption}

\usepackage{mathtools,enumerate,xcolor,geometry}
\mathtoolsset{showonlyrefs=true} 
\theoremstyle{plain}
\newtheorem{theorem}{Theorem}

\newtheorem{lemma}{Lemma}
\newtheorem{corollary}{Corollary}
\newtheorem{definition}{Definition}

\newtheorem{conjecture}{Conjecture}

\newcommand{\N}{\mathbb{N}}
\newcommand{\R}{\mathbb{R}}

\newcommand{\EE}{\mathbb{E}}
\newcommand{\bb}[1]{\boldsymbol{#1}}
\newcommand{\rd}{{\rm d}}

\allowdisplaybreaks

\begin{document}

\begin{frontmatter}

\title{Miscellaneous results related to the Gaussian product inequality conjecture\\for the joint distribution of traces of Wishart matrices}%

\author[a1]{Christian Genest\texorpdfstring{}{)}}\ead{christian.genest@mcgill.ca}%
\author[a1,a2]{Fr\'ed\'eric Ouimet\texorpdfstring{}{)}}\ead{frederic.ouimet@umontreal.ca}%

\address[a1]{Department of Mathematics and Statistics, McGill University, Montr\'eal (Qu\'ebec) Canada H3A 0B9}%
\address[a2]{Centre de recherches math\'ematiques, Universit\'e de Montr\'eal, Montr\'eal (Qu\'ebec) Canada H3T 1J4}%

\begin{abstract}
This note reports partial results related to the Gaussian product inequality (GPI) conjecture for the joint distribution of traces of Wishart matrices. In particular, several GPI-related results from \citet{MR3278931} and \citet{MR3608204} are extended in two ways: by replacing the power functions with more general classes of functions, and by replacing the usual Gaussian and multivariate gamma distributional assumptions by the more general trace-Wishart distribution assumption. These findings suggest that a Kronecker product form of the GPI holds for diagonal blocks of any Wishart distribution.
\end{abstract}

\begin{keyword}
Bernstein function \sep complete monotonicity \sep Gaussian product inequality \sep Kronecker product \sep Laplace exponent \sep multivariate gamma distribution \sep multivariate Laplace transform order \sep trace \sep Wishart distribution
\MSC[2020]{Primary: 60E15; Secondary 26A48, 44A10, 62E15, 62H10, 62H12}
\end{keyword}

\end{frontmatter}

\section{Introduction\label{sec:1}}

Let $\bb{Z} = (Z_1,\dots, Z_d)$ be a centered Gaussian random vector in dimension $d \ge 2$. The strong Gaussian product inequality (GPI) conjecture states that for all reals $\alpha_1, \ldots, \alpha_d \in [0, \infty)$ and any integer $d_1\in \{1, \ldots, d-1 \}$, one has
\begin{equation}
\label{eq:1}
\EE \left(\prod_{i=1}^d |Z_i|^{2 \alpha_i}\right) \geq \EE \left(\prod_{i=1}^{d_1} |Z_i|^{2 \alpha_i}\right) \EE \left(\prod_{i=d_1+1}^d |Z_i|^{2 \alpha_i}\right).
\end{equation}

This inequality is known to hold whenever $|\bb{Z}| = (|Z_1|, \ldots, |Z_d|)$ is multivariate totally positive of order $2$ ($\mathrm{MTP}_2$) on the set $[0,\infty)^d$; see Corollary~1.1 of \citet{MR628759} and references therein. In particular, it is always true in dimension $d = 2$ because if $(Z_1, Z_2)$ is a Gaussian random pair, then $(|Z_1|, |Z_2|)$ is necessarily $\mathrm{MTP}_2$ by Remark~1.4 of \citet{MR628759}. \citet{MR4471184} further showed, using a moment formula from \citet{MR0045347}, that the reverse inequality in~\eqref{eq:1} is valid when $d = 2$ and $(\alpha_1,\alpha_2)\in (-1,0] \times [0,\infty)$, thereby completing the study of the GPI conjecture in the bivariate case.

In general dimension $d \ge 3$, the conditions under which inequality~\eqref{eq:1} holds are still unknown, which justifies the (loose) use of the term conjecture. While counterexamples exist when $d = 3$ both when $\bb{Z}$ is singular \citep{MR4052574} or not \citep{arXiv:2204.06220}, \citet{MR4445681} proved inequality~\eqref{eq:1} for all $\alpha_1, \ldots, \alpha_d \in \N_0 = \{ 0 \} \cup \N = \{0, 1, \ldots \}$ when the covariance matrix only has nonnegative entries, using an Isserlis--Wick type formula due to \citet{MR3324071,arXiv:1705.00163}; see also Corollary~1 of \citet{arXiv:2202.00189} and Remark~2.3 of \citet{arXiv:2204.06220}.

In their paper, \citet{arXiv:2204.06220} also established the analog of inequality~\eqref{eq:1} for all $\alpha_1, \ldots, \alpha_d \in \N_0$ when the variables $Z_1^2, \ldots, Z_d^2$ are replaced by the components of a random vector $\bb{X} = (X_1, \ldots, X_d)$ with multivariate gamma distribution in the sense of \citet{MR44790} whose covariance matrix $\Sigma$ is such that there exists a signature matrix $S$ for which all entries of $S\Sigma S$ are nonnegative.

Put differently, the random vector $\bb{X}$ in the work of \citet{arXiv:2204.06220} consists of the diagonal elements of a $\mathrm{Wishart}_d (2\alpha, \Sigma/2)$ matrix with $2 \alpha \in \N \cup (d - 1, \infty)$ and a symmetric positive semidefinite (SPSD) matrix~$\Sigma$. See, e.g., \citet{MR3325368} for a survey of the properties of this class of distributions. In particular, the Laplace transform of $\bb{X}$ is given, for all vectors $\bb{t} \in [0,\infty)^d$, by
\begin{equation}
\label{eq:2}
\EE \big( e ^{-\bb{t}^{\top} \bb{X}} \big) = | \mathrm{I}_d + \mathrm{diag} (\bb{t}) \, \Sigma |^{-\alpha} ,
\end{equation}
where $\mathrm{I}_d$ denotes the identity matrix of size $d \times d$. For a possible extension of the range of admissible values for the parameter $2\alpha$ to $\N \cup (\lfloor (d - 1)/2 \rfloor, \infty)$, see \citet{arXiv:1606.04747}.

\citet{MR3278931} showed that, for any reals $\beta_1, \ldots, \beta_d \in (-\alpha, 0]$ and integer $d_1\in \{1, \ldots, d-1 \}$, one has
\begin{equation}
\label{eq:3}
\EE \left(\prod_{i=1}^d X_i^{\beta_i} \right) \geq \EE \left(\prod_{i=1}^{d_1} X_i^{\beta_i}\right) \EE \left(\prod_{i=d_1+1}^d X_i^{\beta_i}\right).
\end{equation}
This result was recovered by \citet{arXiv:2204.06220} using the multivariate gamma extension of the Gaussian correlation inequality, due to \citet{MR3289621}. These authors pointed out that the component-wise absolute negative powers of multivariate gamma random vectors are strongly positive upper orthant dependent and then integrated on both sides of the corresponding inequality.

This paper extends the validity of inequalities \eqref{eq:1} and \eqref{eq:3} in various ways. The joint distribution of traces of Wishart matrices is first defined in Section~\ref{sec:2}. It is then shown in Section~\ref{sec:3} that inequality \eqref{eq:3} is valid in arbitrary dimension $d \in \N$ for this class of distributions and for more general functions than powers. A similar result is proved in Section~\ref{sec:4} for inequality~\eqref{eq:1} in dimension $d = 2$. Finally, Section~\ref{sec:5} describes the broader context within which this research fits.

\section{The joint distribution of traces of Wishart matrices\label{sec:2}}

Let $2\alpha \in \N \cup (p - 1, \infty)$ and let $\Sigma$ be an SPSD matrix of size $p\times p$ for some integer $p \in \N = \{1, 2, \ldots \}$. A name is proposed below for  the joint distribution of a random vector of dimension $d \ge 2$ whose components are the traces of diagonal blocks of a $\mathrm{Wishart}_p(2\alpha,\Sigma/2)$ matrix. This distribution was previously studied by \citet{MR263201} in the context of the Gaussian correlation inequality conjecture, later proved in full generality by \citet{MR3289621}.

\begin{definition}[Multivariate trace-Wishart distribution]
\label{def:1}
A random vector $\bb{X}$ of dimension $d \ge 2$ is said to follow a multivariate trace-Wishart distribution with parameters $2\alpha \in \N \cup (p - 1, \infty)$, $p_1, \ldots, p_d \in \N$, and SPSD matrix $\Sigma$ of size $p \times p$ with $p = p_1 + \dots + p_d$, hereafter denoted $\bb{X}\sim \mathrm{TW}_{p_1,\dots,p_d}(\alpha,\Sigma)$, if and only if, for all vectors $\bb{t} \in [0,\infty)^d$,
\begin{equation}
\label{eq:4}
\EE \big( e ^{-\bb{t}^{\top} \bb{X}} \big) = | \mathrm{I}_d + \mathrm{diag}(t_1 \mathrm{I}_{p_1}, \dots, t_d \mathrm{I}_{p_d}) \, \Sigma |^{-\alpha},
\end{equation}
where $\Sigma$ is partitioned in blocks $\Sigma_{ij}$ of size $p_i \times p_j$ for every integers $i, j \in \{ 1, \ldots, d \}$. The expression for the Laplace transform of this distribution can be deduced from Eq.~$(3.2)$~and~Eq.~$(3.4)$ of \citet{MR263201}.
\end{definition}

Observe that if $W \sim \mathrm{Wishart}_p (2\alpha, \Sigma/2)$ has diagonal blocks $W_{11}, \ldots, W_{dd}$ of size $p_1 \times p_1, \ldots, p_d \times p_d$, respectively, then $\bb{X}$ has the same distribution as the random vector $(\mathrm{tr}(W_{11}), \ldots, \mathrm{tr}(W_{dd}))$, where $\mathrm{tr} (\cdot)$ denotes the trace operator. In the special case $p_1 = \dots = p_d = 1$, the multivariate trace-Wishart distribution is simply the multivariate gamma distribution defined in \eqref{eq:2}.

Given the expression of the Laplace transform of the trace-Wishart distribution in \eqref{eq:4}, one can see that all the marginal distributions are embedded, i.e., if $\mathcal{J} $ is a nonempty subset of $ \{1, \ldots, d\}$, then
\begin{equation}
\label{eq:5}
\bb{X}_{\mathcal{J}} = (X_j)_{j\in \mathcal{J}}\sim \mathrm{TW}_{(p_j)_{j\in \mathcal{J}}}(\alpha,\Sigma_{\mathcal{J}}),
\end{equation}
where $\Sigma_{\mathcal{J}}$ is constructed from $\Sigma$ by keeping only the rows and columns indexed by the numbers in $\mathcal{J}$.

Numerical experiments suggest that some weakened form of inequality~\eqref{eq:1} should hold for $\bb{X}\sim \mathrm{TW}_{p_1,\dots,p_d}(\alpha,\Sigma)$, or equivalently for the traces of diagonal blocks of $\mathrm{Wishart}_p (2\alpha, \Sigma/2)$ matrices with $p = p_1 + \dots + p_d$, if $\Sigma$ is symmetric positive definite (SPD). In particular, this conjecture holds true when the entries of $\Sigma$ are nonnegative (even if $\Sigma$ is singular), by a straightforward adaptation of the argument presented in the proof of Theorem~2.1 of \citet{arXiv:2204.06220}.

\begin{conjecture}
\label{con:1}
Let $W$ be a random matrix distributed as $\mathrm{Wishart}_p (2\alpha, \Sigma/2)$ for some $2\alpha \in \N \cup (p - 1, \infty)$ and SPD matrix $\Sigma$ of size $p \times p$ for some integer $p \in \N$. Let $p_1, \ldots, p_d \in \N$ be integers such that $p_1 + \cdots + p_d = p$ and assume that the matrix $\Sigma$ is partitioned in blocks $\Sigma_{ij}$ of size $p_i \times p_j$ for every integers $i, j \in \{ 1, \ldots, d \}$. For each integer $i \in \{ 1, \ldots, d\}$, let $W_{ii}$ denote the $i$th diagonal block of size $p_i \times p_i$ within $W$. It is conjectured that for every integer $d_1 \in \{1, \ldots, d-1 \}$ and all reals $\alpha_1, \ldots, \alpha_d \in [0, \infty)$, one has
\[
\EE \left[\prod_{i=1}^d \left\{\mathrm{tr}(W_{ii})\right\}^{\alpha_i}\right] \geq \EE \left[\prod_{i=1}^{d_1} \left\{\mathrm{tr}(W_{ii})\right\}^{\alpha_i}\right] \EE \left[\prod_{i=d_1+1}^d \left\{\mathrm{tr}(W_{ii})\right\}^{\alpha_i}\right].
\]
\end{conjecture}

If the exponents are restricted to be nonnegative integers in the above expression, Conjecture~\ref{con:1} reduces to the claim that, for all integers $n_1, \ldots, n_d\in \N_0$ and any integer $d_1\in \{1, \ldots, d-1 \}$, one has
\[
\mathrm{tr}\left[ \EE \left\{\otimes_{i=1}^d \left(W_{ii}\right)^{\otimes n_i}\right\} - \EE \left\{ \otimes_{i=1}^{d_1} \left(W_{ii}\right)^{\otimes n_i}\right\} \otimes \EE \left\{\otimes_{i=d_1+1}^d \left(W_{ii}\right)^{\otimes n_i}\right\}\right] \geq 0,
\]
where $\otimes$ denotes the Kronecker product, and $A^{\otimes n}$ means $A \otimes \dots \otimes A$, where $A$ appears $n$ times. To see this, it suffices to use the linearity of the trace operator and the identity $\mathrm{tr}(A) \mathrm{tr}(B) = \mathrm{tr}(A \otimes B)$.

\section{Extension of a result of \texorpdfstring{\citet{MR3278931}}{Wei (2014)} concerning inequality~(3)\label{sec:3}}

As mentioned in the introduction, \citet{MR3278931} established inequality~\eqref{eq:3} for a random vector $\bb{X} = (X_1, \ldots, X_d)$ with Laplace transform \eqref{eq:2}, arbitrary integer $d_1 \in \{1, \ldots, d-1 \}$, and any reals $\beta_1, \ldots, \beta_d \in (-\alpha, 0]$; see the lower bound in Theorem~3.2 of~\citet{MR3278931} and Section~4 of~\citet{arXiv:2204.06220}.

The following result extends Wei's finding by showing that inequality \eqref{eq:3} holds for the more general multivariate trace-Wishart distribution and that the negative powers in \eqref{eq:3} can be replaced by arbitrary completely monotone functions. Recall from Definition~1.3 of \citet{MR2978140} that a function $\phi: (0, \infty)\to \R$ is completely monotone if $\phi$ is nonnegative, infinitely differentiable on $(0, \infty)$, and satisfies $(-1)^n \phi^{(n)}(t) \geq 0$ for every integer $n\in \N$ and every real $t \in (0, \infty)$.

\begin{theorem}\label{thm:1}
Let $\bb{X}$ be a random vector of dimension $d \ge 2$ distributed as $\mathrm{TW}_{p_1,\dots,p_d}(\alpha, \Sigma)$ for some $2\alpha \in \N \cup (p - 1, \infty)$, integers $p_1, \ldots, p_d \in \N$, and SPSD matrix $\Sigma$ of size $p \times p$ with $p = p_1 + \dots + p_d$. Then for any collection $\phi_1,\dots,\phi_d$ of completely monotone functions on $(0, \infty)$ and any integer $d_1\in \{1,\dots,d-1\}$, one has
\[
\EE\left\{\prod_{i=1}^d \phi_i(X_i)\right\} \geq \EE\left\{\prod_{i=1}^{d_1} \phi_i(X_i)\right\} \EE\left\{\prod_{i=d_1+1}^{d} \phi_i(X_i)\right\},
\]
provided that the expectations exist.
\end{theorem}

The proof of this result relies on the notion of multivariate Laplace transform order, whose definition is recalled below from p.~350 of the book by~\citet{MR2265633}; see also \citet{MR1857969}.

\begin{definition}[Multivariate Laplace transform order]
\label{def:2}
Let $\bb{X}$ and $\bb{Y}$ be two $d$-variate random vectors with nonnegative entries. Then $\bb{X}$ is said to be smaller than $\bb{Y}$ in the multivariate Laplace transform order, hereafter denoted by $\bb{X} \preceq_{\mathrm{Lt}} \bb{Y}$, if and only if, for all vectors $\bb{t} \in (0, \infty)^d$, one has $\EE \big ( e^{-\bb{t}^{\top} \bb{X}} \big) \geq \EE \big (e ^{-\bb{t}^{\top} \bb{Y}} \big )$.
\end{definition}

\begin{lemma}
\label{lemma:1}
Let $\bb{X}$ and $\bb{X}^{\star}$ be two random vectors of dimension $d \ge 2$ such that for some $2\alpha \in \N \cup (p - 1, \infty)$ and integers $p_1, \ldots, p_d \in \N$, one has
\[
\bb{X} \sim \mathrm{TW}_{p_1, \ldots, p_d} (\alpha,\Sigma), \quad \bb{X}^{\star} \sim \mathrm{TW}_{p_1, \ldots, p_d} (\alpha,\Sigma^{\star})
\]
for some SPSD matrices $\Sigma$ and $\Sigma^{\star}$ of size $p \times p$ with $p = p_1 + \cdots + p_d$. Suppose that for some integer $d_1 \in \{ 1, \ldots, d-1\}$,
\[
\Sigma =
\begin{pmatrix}\Sigma_{11} & \Sigma_{12} \\ \Sigma_{21} & \Sigma_{22}
\end{pmatrix}, \quad
\Sigma^{\star} = \mathrm{diag}(\Sigma_{11},\Sigma_{22}) =
\begin{pmatrix}
\Sigma_{11} & 0_{q_1 \times q_2} \\ 0_{q_2 \times q_1} & \Sigma_{22}
\end{pmatrix},
\]
where $0_{a \times b}$ refers to an $a \times b$ matrix of zeros, $\Sigma_{11}$ is a matrix of size $q_1 \times q_1$ with $q_1 = p_1 + \dots + p_{d_1}$, and $\Sigma_{22}$ is a matrix of size $q_2 \times q_2$ with $q_2 = p_{d_1+1} + \dots + p_d$. Then $\bb{X} \preceq_{\mathrm{Lt}} \bb{X}^{\star}$.
\end{lemma}

\begin{proof}[\bf Proof of Lemma~\ref{lemma:1}]
By Definition~\ref{def:1}, one has
\[
\bb{X} \preceq_{\mathrm{Lt}} \bb{X}^{\star} \quad \Longleftrightarrow \quad \forall_{t_1, \ldots, t_d \in (0, \infty)} ~~ |\mathrm{I}_d + \mathrm{diag}(t_1 \mathrm{I}_{p_1}, \dots, t_d \mathrm{I}_{p_d}) \, \Sigma |^{-\alpha} \geq  |\mathrm{I}_d + \mathrm{diag}(t_1 \mathrm{I}_{p_1}, \dots, t_d \mathrm{I}_{p_d}) \, \Sigma^{\star} |^{-\alpha} .
\]
Equivalently, the right-hand inequality can be written as
\[
|\mathrm{I}_d + \mathrm{diag}(t_1 \mathrm{I}_{p_1}, \dots, t_d \mathrm{I}_{p_d}) \, \Sigma | \leq | \mathrm{I}_d + \mathrm{diag}(t_1 \mathrm{I}_{p_1}, \dots, t_d \mathrm{I}_{p_d}) \, \Sigma^{\star} |,
 \]
which, taking into account the block diagonal structure of $\Sigma^{\star}$, is equivalent to
\[
|\mathrm{I}_d + \mathrm{diag}(t_1 \mathrm{I}_{p_1}, \dots, t_d \mathrm{I}_{p_d}) \, \Sigma | \leq | \mathrm{I}_d + \mathrm{diag}(t_1 \mathrm{I}_{p_1}, \dots, t_{d_1} \mathrm{I}_{p_{d_1}}) \, \Sigma_{11}  |  | \mathrm{I}_d + \mathrm{diag}(t_{d_1+1} \mathrm{I}_{p_{d_1+1}}, \dots, t_d \mathrm{I}_{p_d}) \, \Sigma_{22}  |.
\]
The latter fact is valid by Fischer's inequality; see, e.g., Theorem~7.8.5 in \citet{MR2978290}.
\end{proof}

\begin{proof}[\bf Proof of Theorem~\ref{thm:1}]
Let $\bb{X} \sim \mathrm{TW}_{p_1, \ldots, p_d}(\alpha,\Sigma)$ for some $2\alpha \in \N \cup (p - 1, \infty)$, integers $p_1, \ldots, p_d \in \N$, and SPSD matrix $\Sigma$ of size $p\times p$ with $p = p_1 + \dots + p_d$. Fix an integer $d_1 \in \{1, \ldots, d-1 \}$ and let $q_1 = p_1 + \dots + p_{d_1}$, $q_2 = p_{d_1+1} + \dots + p_d$ so that $q_1 + q_2 = p$. Now write
\[
\Sigma = \begin{pmatrix}
\Sigma_{11} & \Sigma_{12} \\ \Sigma_{21} & \Sigma_{22}
\end{pmatrix}, \quad \Sigma^{\star} = \mathrm{diag}(\Sigma_{11},\Sigma_{22}) =
\begin{pmatrix}
\Sigma_{11} & 0_{q_1 \times q_2} \\ 0_{q_2 \times q_1} & \Sigma_{22}
\end{pmatrix},
\]
where $\Sigma_{11}$ and $\Sigma_{22}$ are $q_1 \times q_1$ and $q_2 \times q_2$, respectively. Next consider a $d$-variate random vector $\bb{X}^{\star}$ distributed as $\mathrm{TW}_{p_1, \ldots, p_d} (\alpha, \Sigma^{\star})$. In view of Lemma~\ref{lemma:1}, one has $\bb{X} \preceq_{\mathrm{Lt}} \bb{X}^{\star}$  and hence Theorem~7.D.6 of~\citet{MR2265633} implies that, for any collection $\phi_1, \ldots, \phi_d$ of completely monotone functions on the interval $(0, \infty)$, one has
\[
\EE \left\{ \prod_{i=1}^d \phi_i(X_i)\right\} \geq \EE \left\{ \prod_{i=1}^{d_1} \phi_i(X_i^{\star})\right\} \EE \left\{ \prod_{i=d_1+1}^{d} \phi_i(X_i^{\star})\right\}.
\]
To conclude, note that in view of property~\eqref{eq:5}, the random vectors $(X_1,\dots,X_{d_1})$ and $(X_{d_1+1},\dots,X_d)$ have the same distribution as $(X_1^{\star}, \ldots, X_{d_1}^{\star})$ and $(X_{d_1+1}^{\star}, \ldots, X_d^{\star})$, respectively.
\end{proof}

The following is an interesting consequence of Theorem~\ref{thm:1}.

\begin{corollary}
\label{cor:1}
Let $\bb{X}$ be a random vector of dimension $d \ge 2$ distributed as $\mathrm{TW}_{p_1,\dots,p_d}(\alpha, \Sigma)$ for some $2\alpha \in \N \cup (p - 1, \infty)$, integers $p_1, \ldots, p_d \in \N$, and SPSD matrix $\Sigma$ of size $p \times p$ with $p = p_1 + \dots + p_d$. Then for any integer $d_1 \in \{1, \ldots, d-1\}$, the following statements hold true.
\begin{enumerate}[(a)]
\item For any reals $q_1, \ldots, q_d \in (-\alpha,0]$,
\begin{equation}
\label{eq:6}
\EE \left( \prod_{i=1}^d X_i^{q_i} \right) \geq \EE \left( \prod_{i=1}^{d_1} X_i^{q_i} \right) \EE \left( \prod_{i=d_1+1}^d X_i^{q_i} \right).
\end{equation}

\item
For any reals $r_1, \ldots, r_d \in [0,1]$,
\begin{equation}
\label{eq:7}
\EE \left( \prod_{i=1}^d e^{-X_i^{r_i}} \right) \geq \EE \left( \prod_{i=1}^{d_1} e^{-X_i^{r_i}} \right) \EE \left( \prod_{i=d_1+1}^d e^{-X_i^{r_i}} \right).
\end{equation}

\item For any integer $n\in \N$ and any sequence $A_1, \ldots, A_d$ of SPSD matrices of size $n \times n$,
\begin{equation}
\label{eq:8}
\mathrm{tr}\left[ \EE \left\{ \exp \left( -\oplus_{i=1}^d X_i A_i \right) \right\} - \EE \left\{ \exp\left(-\oplus_{i=1}^{d_1} X_i A_i \right) \right\} \otimes \EE \left\{ \exp \left(-\oplus_{i=d_1+1}^d X_i A_i \right) \right\} \right] \geq 0,
\end{equation}
where $\oplus$ denotes the Kronecker sum.
\end{enumerate}
\end{corollary}

\begin{proof}[\bf Proof of Corollary~\ref{cor:1}]
Note that the map $t \mapsto t^q$ is completely monotone on the interval $(0, \infty)$ for every real $q \in (-\infty, 0]$. Therefore, claim \eqref{eq:6} is a consequence of Theorem~\ref{thm:1}. The restriction on the reals $q_1, \ldots, q_d$ in the statement of part~(a) is simply there to ensure that the expectations exist.

Next, fix a real $r \in [0,1]$ and consider the map $g_r: (0, \infty) \to (0, \infty)$ defined for all real $t \in (0, \infty)$ by $g_r(t) = \exp(-t^r)$. Observe that $g_r$ is logarithmically completely monotone, i.e., such that $-(\ln  g_r)^\prime$ is completely monotone on the interval $(0, \infty)$; see Definition~1 of~\citet{MR2075188} or~Definition~5.10 of~\citet{MR2978140}. It is well known that the property of being logarithmically completely monotone is stronger than being completely monotone; see, e.g., Theorem~1.1 of~\citet{MR2112748} or Theorem~1 of~\citet{MR2075188}. Therefore, $g_r$ is completely monotone on the interval $(0, \infty)$ for any real $r \in [0, 1]$. Accordingly, claim \eqref{eq:7} is a simple consequence of Theorem~\ref{thm:1}.

Turning to claim~\eqref{eq:8}, it is known from a result proven by \citet{MR3143891}, formerly known as the BMV conjecture, that for any SPSD matrix $A$ of size $n\times n$, the map $t \mapsto \mathrm{tr} \{ \exp(-t A) \}$ is completely monotone on $(0,\infty)$. Therefore, if $\otimes$ denotes the Kronecker product, one has, for arbitrary integer $d_1 \in \{ 1, \ldots, d-1 \}$,
\begin{align*}
\mathrm{tr}\left[\EE\left\{\exp\left(-\oplus_{i=1}^d X_i A_i\right)\right\}\right]
&= \EE\left[\mathrm{tr}\left\{\exp\left(-\oplus_{i=1}^d X_i A_i\right)\right\}\right] \\
&= \EE\left[\mathrm{tr}\left\{\otimes_{i=1}^d \exp\left(-X_i A_i\right)\right\}\right]
= \EE\left[\prod_{i=1}^d \mathrm{tr}\left\{\exp\left(-X_i A_i\right)\right\}\right]  \\
&\geq \EE\left[\prod_{i=1}^{d_1} \mathrm{tr}\left\{\exp\left(-X_i A_i\right)\right\}\right] \EE\left[\prod_{i=d_1+1}^d \mathrm{tr}\left\{\exp\left(-X_i A_i\right)\right\}\right] ,
\end{align*}
where the inequality is a direct application of Theorem~\ref{thm:1} with $\phi_i(x) = \mathrm{tr} \{\exp(- x A_i) \}$. The lower bound can then be expressed in the alternative forms
\vspace{-2mm}
\begin{align*}
\EE \left[ \mathrm{tr} \left\{\otimes_{i=1}^{d_1} \exp \left( -X_i A_i \right) \right\} \right] & \EE \left[ \mathrm{tr} \left\{ \otimes_{i=d_1+1}^d \exp \left( -X_i A_i \right) \right\} \right] \\[2mm]
& \hspace{-3mm} = \EE \left[ \mathrm{tr} \left\{ \exp \left(- \oplus_{i=1}^{d_1} X_i A_i \right) \right\} \right] \EE \left[ \mathrm{tr} \left\{ \exp \left( -\oplus_{i=d_1+1}^d X_i A_i \right) \right\} \right] \\[2mm]
& \hspace{-3mm} = \mathrm{tr} \left[\EE \left\{ \exp \left(- \oplus_{i=1}^{d_1} X_i A_i \right) \right\} \right] \mathrm{tr} \left[ \EE \left\{ \exp \left( -\oplus_{i=d_1+1}^d X_i A_i \right) \right\} \right] \\[2mm]
& \hspace{-3mm}  = \mathrm{tr} \left[ \EE \left\{ \exp \left(- \oplus_{i=1}^{d_1} X_i A_i \right) \right\} \otimes \EE \left\{ \exp \left( -\oplus_{i=d_1+1}^d X_i A_i \right) \right\} \right].
\end{align*}
This establishes claim~\eqref{eq:8} and concludes the proof of Corollary~\ref{cor:1}.
\end{proof}

Corollary~\ref{cor:1}~(a) extends Theorem~3.2 of~\citet{MR3278931} by showing that inequality~\eqref{eq:3} with negative powers holds for vectors $\bb{X}$ with multivariate trace-Wishart distribution, and not only in the special case of the multivariate gamma distribution in the sense of \citet{MR44790}.

Corollary~\ref{cor:1}~(b) extends a result of~\citet{MR3608204}. Corollary~1.1~(ii) in these authors' paper states that if $\bb{Z} = (Z_1, \ldots, Z_d)$ is a centered Gaussian random vector and $r_1, \ldots, r_d \in [0,1/2]$ are any reals, then
\begin{equation}
\label{eq:9}
\EE \left( \prod_{i=1}^d e^{-|Z_i|^{2 r_i} } \right) \geq \prod_{i=1}^{d} \EE \left( e^{-|Z_i|^{2 r_i}} \right) .
\end{equation}
Inequality~\eqref{eq:9} can easily be deduced through successive applications of inequality~\eqref{eq:7} by taking $2\alpha = 1$, $p_1 = \dots = p_d = 1$ and by replacing $X_1, \ldots, X_d$ with $Z_1^2, \ldots, Z_d^2$. In fact, note that the range of the exponents $r_1, \ldots, r_d$ in inequality~\eqref{eq:7} is twice as large as their range in inequality~\eqref{eq:9}. Therefore, inequality~\eqref{eq:9} is not only valid for the Gaussian distribution but also for the much more general multivariate trace-Wishart distribution and for a larger range of exponents than derived by \citet{MR3608204}.

Finally, Corollary~\ref{cor:1}~(c) is a matrix-variate extension of Lemma~\ref{lemma:1}, from the initial case $n = 1$ with $A_i = (t_i)_{1\times 1}$ for every integer $i \in \{1, \ldots, d \}$ to a sequence $A_1, \ldots, A_d$ of SPSD matrices of size $n \times n$ for any integer $n \in \N$.

\section{Extension of a result of  \texorpdfstring{\citet{MR3608204}}{Liu et al. (2017)} concerning inequality~(1)\label{sec:4}}

\citet{MR3608204}, whose work was mentioned earlier in relation to inequality~\eqref{eq:9}, state at the bottom of p.~2 of their paper that if $(Z_1, Z_2)$ is a centered Gaussian random pair, then, for all reals $q_1, q_2 \in [0,1]$, one has
\begin{equation}
\label{eq:10}
\EE ( |Z_1|^{2 q_1} |Z_2|^{2 q_2} ) \geq \EE ( |Z_1|^{2 q_1} ) \EE (|Z_2|^{2 q_2} ).
\end{equation}
This is a special case of the strong Gaussian product inequality conjecture, inequality~\eqref{eq:1}. If the covariance matrix is singular, then an even stronger inequality holds; see Proposition~3.1~(ii) of \citet{MR4471184}. The $d$-dimensional analog of inequality \eqref{eq:10} was shown to hold by \citet{MR2385646} for $q_1 = \dots = q_d = 1$.

In this section, inequality~\eqref{eq:10} is extended in multiple ways using the notion of Bernstein function recalled below from Definition~3.1 and Theorem~3.2 of~\citet{MR2978140}.

\begin{definition}[Bernstein function]
A map $f: [0, \infty) \to \R$ is said to be a Bernstein function (or Laplace exponent) on the interval $[0, \infty)$ if and only if there exist reals $a, b \in [0, \infty)$ and a measure $\mu$ on $(0, \infty)$ with $\int \min (1, t) \mu(\rd t) < \infty$ such that, for every real $\lambda \in [0, \infty)$,
$$
f (\lambda) = a + b \lambda + \int_{(0,\infty)} (1 - e^{-\lambda t}) \mu(\rd t).
$$
Such a map $f$ is referred to as the Bernstein function with triplet $(a, b, \mu)$, where the conditions on $a$, $b$ and $\mu$ are kept implicit.
\end{definition}

Equivalently, a map $f: [0, \infty) \to \R$ is a Bernstein function if and only if $f$ is nonnegative, continuous on $[0, \infty)$, infinitely differentiable on $(0, \infty)$, and its first derivative $f^\prime$ is completely monotone on $(0, \infty)$, i.e., for every integer $n \in \N$ and every real $\lambda \in (0, \infty)$, one has $(-1)^{n-1} f^{(n)}(\lambda) \geq 0$.

The following result is akin to Theorem~5.A.4 in \citet{MR2265633} but seemingly new.

\begin{theorem}
\label{thm:2}
Let $(X_1, X_2)$ and $(X_1^{\star}, X_2^{\star})$ be two random pairs with nonnegative entries and identical margins. Suppose that the random variables $X_1^{\star}$ and $X_2^{\star}$ are independent, and that $(X_1, X_2) \preceq_{\mathrm{Lt}} (X_1^{\star}, X_2^{\star})$. If $f$ and $g$ are Bernstein functions with triplets $(a_1,0,\mu_1)$ and $(a_2,0,\mu_2)$, respectively, then
\begin{equation}
\EE \{ f(X_1) g(X_2)  \} \geq \EE  \{ f(X_1^{\star}) \} \EE \{ g(X_2^{\star}) \}.
\end{equation}
\end{theorem}

\begin{proof}[\bf Proof of Theorem~\ref{thm:2}]
Given that $(X_1, X_2) \preceq_{\mathrm{Lt}} (X_1^{\star}, X_2^{\star})$, one has that, for any given reals $s, t \in [0,\infty)$,
\begin{align}
\label{eq:11}
\EE \big \{ \big( 1 - e^{-s X_1} \big) \big( 1 - e^{-t X_2} \big) \big\}
& = 1 - \EE \big( e^{-s X_1} \big) - \EE \big( e^{-t X_2} \big) + \EE \big( e^{-s X_1} e^{-t X_2} \big) \notag \\
& \geq 1 - \EE \big( e^{-s X_1} \big) - \EE \big( e^{-t X_2} \big) + \EE \big( e^{-s X_1^{\star}} e^{-t X_2^{\star}} \big) \notag \\
& = 1 - \EE \big( e^{-s X_1^{\star}} \big) - \EE \big( e^{-t X_2^{\star}} \big) + \EE \big( e^{-s X_1^{\star}} \big) \EE \big( e^{-t X_2^{\star}} \big) \notag \\
&= \EE \big\{ \big( 1 - e^{-s X_1^{\star}} \big) \big\} \EE \big\{ \big( 1 - e^{-t X_2^{\star}} \big) \big\},
\end{align}
owing to the fact that the pairs $(X_1, X_2)$ and $(X_1^{\star}, X_2^{\star})$ have identical margins and that the random variables $X_1^{\star}$ and $X_2^{\star}$ are independent. Next observe that if $f$ and $g$ are Bernstein functions as specified in the statement of the theorem, then the expectation $\EE \left\{ f(X_1) g(X_2) \right\}$ can be written as
\begin{multline*}
a_1 a_2 + a_1 \int_{(0,\infty)} \EE \left(1 - e^{-t X_2} \right) \mu_2(\rd t) + a_2 \int_{(0,\infty)} \EE \left( 1 - e^{-s X_1} \right) \mu_1(\rd s) \\
+ \int_{(0,\infty)^2} \EE \left\{ \left( 1 - e^{-s X_1} \right) \left(1 - e^{-t X_2} \right) \right\} \mu_1(\rd s) \mu_2(\rd t).
\end{multline*}
Given that the random pairs $(X_1, X_2)$ and $(X_1^{\star}, X_2^{\star})$ have identical margins by assumption, and in view of inequality~\eqref{eq:11}, the above expression is bounded from below by
\begin{multline*}
a_1 a_2 + a_1 \int_{(0,\infty)} \EE \big( 1 - e^{-t X_2^{\star}} \big) \mu_2(\rd t) + a_2 \int_{(0,\infty)} \EE \big( 1 - e^{-s X_1^{\star}} \big) \mu_1(\rd s) \\
+ \int_{(0,\infty)^2} \EE \big( 1 - e^{-s X_1^{\star}} \big) \EE \big( 1 - e^{-t X_2^{\star}} \big) \mu_1(\rd s) \mu_2(\rd t),
\end{multline*}
which is the same as $\EE \left\{ f(X_1^{\star}) \right\} \EE \left\{ g(X_2^{\star}) \right\} $. Therefore, the claim is proved.
\end{proof}

As shown next, Theorem~\ref{thm:2} makes it possible to relax the distributional assumption in inequality~\eqref{eq:10}. However, it remains an open question whether a weak version of the GPI, Eq.~\eqref{eq:10}, is true or false in dimension $d \geq 3$, and whether or not there exists a higher dimensional version of Theorem~\ref{thm:2}.

\begin{corollary}
\label{cor:2}
Let $(X_1, X_2)$ be a random pair with distribution $\mathrm{TW}_{p_1,p_2} (\alpha,\Sigma)$ for some $2\alpha \in \N \cup (p - 1, \infty)$, integers $p_1$, $p_2\in \N$, and SPSD matrix $\Sigma$ of size $p\times p$ with $p = p_1 + p_2$. Then, for all reals $q_1,q_2\in [0,1]$, one has
\begin{equation}
\label{eq:12}
\EE (X_1^{q_1} X_2^{q_2} ) \geq \EE (X_1^{q_1} ) \EE (X_2^{q_2} ).
\end{equation}
\end{corollary}

\begin{proof}[\bf Proof of Corollary~\ref{cor:2}]
Let $\Sigma_{11}$ and $\Sigma_{22}$ be the two diagonal blocks of size $p_1 \times p_1$ and $p_2 \times p_2$ within $\Sigma$. Let $(X_1^{\star}, X_2^{\star})$ be a random pair with distribution $\mathrm{TW}_{p_1,p_2}(\alpha, \Sigma^{\star})$ with $\Sigma^{\star} = \mathrm{diag} (\Sigma_{11}, \Sigma_{22})$. It is known from Lemma~\ref{lemma:1} that $(X_1, X_2) \preceq_{\mathrm{Lt}} (X_1^{\star},X_2^{\star})$. Moreover, for any real $q \in [0,1]$, the map $t \mapsto t^q$ is a Bernstein function on $[0, \infty)$ with $a = b = 0$ and some nonnegative measure $\mu$; see, e.g., Eq.~(1) of \citet{MR2978140}. Therefore, the claim \eqref{eq:12} is a consequence of Theorem~\ref{thm:2}.
\end{proof}

\section{Closing comments\label{sec:5}}

The strong Gaussian product inequality conjecture, which appears in~\eqref{eq:1}, is an extension of a weaker conjecture to the effect that, for any centered Gaussian random vector $\bb{Z} = (Z_1, \ldots, Z_d)$ and all nonnegative reals $\alpha_1, \ldots, \alpha_d \in [0, \infty)$, one has
\begin{equation}
\label{eq:13}
\EE \left( \prod_{i=1}^d |Z_i|^{2 \alpha_i} \right) \geq \prod_{i=1}^d \EE \left( |Z_i|^{2 \alpha_i} \right).
\end{equation}
Notice that the random vector $\bb{Z}$ can be singular here.

This conjecture was formulated originally by \citet{MR2886380}. It was proved recently by \citet{MR4466643} for all nonnegative integers $\alpha_1, \ldots, \alpha_d \in \N_0$ under the assumption that the covariance matrix $\Sigma$ is completely positive, using a combinatorial approach closely related to the complete monotonicity of multinomial probabilities shown by \citet{MR3825458} and \citet{MR4201158}. The result of Genest and Ouimet was extended shortly after to all covariance matrices with nonnegative entries by \citet{MR4445681}, using an Isserlis--Wick type formula of \citet{MR3324071,arXiv:1705.00163} as already mentioned in Section~\ref{sec:1}.

Inequality~\eqref{eq:13} was also proved in dimension $d = 3$ for all nonnegative integers $\alpha_1 = \alpha_2 \in \N_0$ and $\alpha_3 \in \N_0$ by~\citet{MR4052574}. Their result, labeled Theorem~3.2 in their paper, was derived using a dimension reduction argument and estimates on Gaussian hypergeometric functions. In particular, in the special case $\alpha_1 = \alpha_2 = \alpha_3$, their result proves the $3$-dimensional version of the following even weaker statement: for every integer $m \in \N_0$,
\begin{equation}
\label{eq:14}
\EE \left( \prod_{i=1}^d |Z_i|^{2 m} \right) \geq \prod_{i=1}^d \EE ( |Z_i|^{2 m} ).
\end{equation}

The validity of inequality~\eqref{eq:14} for centered Gaussian random vectors was first conjectured by~\citet{MR1636556} in the context of the real polarization constant problem in functional analysis. It is \citet{MR2385646} who made the connection between the two problems and expressed the conjecture in the form \eqref{eq:14}.

In the form \eqref{eq:14}, the Gaussian product inequality conjecture is known to imply the real polarization constant conjecture \citep{MR3425898}. It is also related to the so-called $U$-conjecture to the effect that if $P$ and $Q$ are two non-constant polynomials on $\mathbb{R}^d$ such that the random variables $P(\bb{Z})$ and $Q(\bb{Z})$ are independent, then there exist an orthogonal transformation $L$ on $\mathbb{R}^d$ and an integer $d_1 \in \{ 1, \ldots , d - 1 \}$ such that $P \circ L$ is a function of $(Z_1, \ldots, Z_{d_1})$ and $Q \circ L$ is a function of $(Z_{d_1+1}, \ldots, Z_d)$; see, e.g., \citet{MR0346969} or \citet{MR3425898} and references therein.

Inequality \eqref{eq:14} for Gaussian random vectors was proved for $m = 1$ in every dimension $d \in \N$ by \citet{MR2385646} and in dimension $d = 3$ for every integer $m \in \N_0$ by \citet{MR4052574} using Gaussian hypergeometric functions; see their Theorem~1.1. Another extension, valid in dimension $d = 3$ and for all $(\alpha_1, \alpha_2, \alpha_3) \in \{1\} \times \{2, 3\} \times \N_0$ was recently obtained by~\citet{MR4445681} using a brute-force combinatorial approach. Finally, using a sums-of-squares approach along with extensive symbolic/numerical computations in \texttt{Macaulay2} and \texttt{Mathematica}, \citet{arXiv:2205.02127} recently stated in their Theorems~4.1~and~4.2 the validity of inequality~\eqref{eq:13} when $(d, \alpha_1, \alpha_2, \alpha_3) \in \{3\} \times \N_0 \times \{3\} \times \{2\}$ and $(d, \alpha_1, \alpha_2, \alpha_3, \alpha_4) \in \{4\} \times \N_0 \times \{2\} \times \{2\} \times \{2\}$, respectively.

\section*{Funding}

Genest's research is funded in part by the Canada Research Chairs Program (Grant no.~950--231937) and the Natural Sciences and Engineering Research Council of Canada (RGPIN--2016--04720). Ouimet benefited from postdoctoral fellowships from the Natural Sciences and Engineering Research Council of Canada and the Fond qu\'eb\'ecois de la recherche -- Nature et technologies (B3X supplement and B3XR). Ouimet is currently supported by a CRM-Simons postdoctoral fellowship from the Centre de recherches math\'ematiques (Montr\'eal, Canada) and the Simons Foundation.


\end{document}